\documentclass{article}
\usepackage{amsmath}
\usepackage{amssymb}
\usepackage{amsthm}
\usepackage{bm}
\usepackage{graphicx}
\usepackage{caption}
\usepackage[margin = 1in]{geometry}
\usepackage{cite}
\usepackage{color}
\usepackage{media9}
\usepackage{subfigure}
\usepackage{comment}
\usepackage[colorlinks,urlcolor=blue,citecolor=blue,linkcolor=blue]{hyperref}
\usepackage[colorinlistoftodos]{todonotes}

\newcommand{\innp}[1]{\left\langle #1 \right\rangle}

\newcommand{\mA}{\mathbf{A}}

\newcommand{\vx}{\mathbf{x}}

\newcommand{\cl}{\mathcal{K}}

\newcommand{\vy}{\mathbf{y}}

\newcommand{\vv}{\mathbf{v}}

\newcommand{\vb}{\mathbf{b}}

\newcommand{\vu}{\mathbf{u}}

\newcommand{\rr}{\mathbb{R}}

\newcommand{\adgt}{\textsc{adgt}}

\graphicspath{{imgs/}}

\makeatletter
\def\mathcolor#1#{\@mathcolor{#1}}
\def\@mathcolor#1#2#3{%
  \protect\leavevmode
  \begingroup
    \color#1{#2}#3%
  \endgroup
}
\makeatother

\newcommand*{\vsepfbox}[1]{%
  \begingroup
    \sbox0{\fbox{#1}}%
    \setlength{\fboxrule}{0pt}%
    \mbox{\kern-\fboxsep\fbox{\unhbox0}\kern-\fboxsep}%
  \endgroup
}

\theoremstyle{plain} \numberwithin{equation}{section}
\newtheorem{theorem}{Theorem}[section]
\numberwithin{theorem}{section}

\newtheorem{lemma}[theorem]{Lemma}
\newtheorem{proposition}[theorem]{Proposition}

\theoremstyle{definition}
\newtheorem{definition}[theorem]{Definition}

\DeclareMathOperator*{\argmin}{argmin}

\usepackage{times}

\title{Conjugate Gradients and Accelerated Methods Unified:\\
The Approximate Duality Gap View}
\author{Jelena Diakonikolas\thanks{Partially supported by the National Science Foundation under Award  \#CCF-1740855.} \\
Department of Computer Sciences,\\
UW-Madison,\\
\texttt{jelena@cs.wisc.edu}
\and Lorenzo Orecchia\thanks{Partially supported by the National Science Foundation under Award \#CCF-1718342.}\\
Department of Computer Science,\\
University of Chicago,\\
\texttt{orecchia@uchicago.edu}}
\date{}

\begin{document}

\maketitle

\begin{abstract}
This note provides a novel, simple analysis of the method of conjugate gradients for the minimization of convex quadratic functions. In contrast with standard arguments, our proof is entirely self-contained and does not rely on the existence of Chebyshev polynomials. Another advantage of our development is that it clarifies the relation between the method of conjugate gradients and general accelerated methods for smooth minimization by unifying their analyses within the framework of the Approximate Duality Gap Technique~\cite{thegaptechnique}.
\end{abstract}

\section{Introduction}

Accelerated methods in first-order convex optimization have long been the subject of fascination among optimization researchers and enthusiasts alike, leading to a large number of different interpretations of the phenomenon of acceleration in recent years (see, e.g.,~\cite{AXGD,Bubeck2015,Scieur2017,wibisono2016variational,AO-survey-nesterov,drusvyatskiy2016optimal,flammarion2015averaging,SuBC16}). The first method that can be considered to achieve acceleration in the blackbox model of first-order optimization\footnote{In this model, an algorithm accesses the function via first-order oracle queries.} is the method of conjugate gradients (CG), due to Hestenes and Stiefel~\cite{Hestenes1952}. This method achieves the optimal convergence rate for the class of unconstrained convex \emph{quadratic} minimization problems\footnote{Equivalently, the method is used for solving positive semidefinite linear systems.}. The optimality of the method was proved in~\cite{nemirovskii1983problem}, which also provided lower bounds for the more general settings of smooth convex and smooth strongly convex minimization.

In a breakthrough result, Nesterov~\cite{Nesterov1983} introduced a method for smooth minimization that achieved the optimal $1/k^2$ convergence rate. In the same work~\cite{Nesterov1983}, it was also shown that for the class of smooth and strongly convex minimization problems, the same method, when coupled with scheduled restarts, converges to a point $\vx$ with $f(\vx) - f(\vx^*) \leq \epsilon$ in $O(\sqrt{\kappa}\log(\frac{f(\vx_0) - f(\vx^*)}{\epsilon}))$ iterations, where $\vx^* = \argmin_{\vx} f(\vx),$ $\vx_0$ is the initial point, and $\kappa$ is the condition number of $f.$ This result is iteration-complexity-optimal~\cite{nemirovskii1983problem} for the class of smooth and strongly convex minimization problems. Further, the results from~\cite{Nesterov1983} generalized to other normed spaces and the setting of constrained minimization (see~\cite{nesterov2018lectures} and references therein). 

Despite this long interesting line of work, the exact relation between CG and generic accelerated methods is still unclear. In particular, standard analyses of CG (e.g.,~\cite{shewchuk1994introduction,vishnoi2013lx}) greatly depart from the analyses of accelerated algorithms and do not reveal the similarity between these approaches. On the other hand, the analyses of accelerated methods are often considered to rely on ``algebraic tricks.'' Even the later-introduced powerful technique of estimate sequences of Nesterov (see, e.g.,~\cite{nesterov2005smooth}) that led to many other important results in optimization, can be difficult to grasp. 

In this note, we seek to unify the analysis of different accelerated methods and highlight through the analysis how they relate to each other. To do so, we rely on the use of the Approximate Duality Gap Technique (\adgt), which frames the design and analysis of first-order methods in terms of the iterative construction of an upper approximation of the optimality gap. \adgt~was introduced by the authors in~\cite{thegaptechnique} to provide a unifying and intuitive analysis of a large class of first-order optimization methods. This technique is closely related to the estimate sequences technique of Nesterov; however, it has the added benefit of being constructive and more intuitive (see~\cite{thegaptechnique} for more details). 

Within this framework, we show that it is possible to relate the steps of the methods of Nesterov~\cite{Nesterov1983} and Nemirovski~\cite{nemirovskii1983problem,nemirovski-line-search-acc,nemirovski-1D-line-search} to those of CG to provide a convergence guarantee on CG. The argument is extremely simple, essentially hinging on the fact that the points queried by CG must be at least as good, in terms of function value, as those that would be queried by Nesterov's or Nemirovski's method. This can be seen as a ``polynomial-free'' analogue of the standard proof of CG, which works by arguing that CG must perform at least as well as Chebyshev iteration.

 \subsection{Related Work}
 
As noted earlier, a significant body of recent research~\cite{AXGD,Bubeck2015,Scieur2017,wibisono2016variational,AO-survey-nesterov,drusvyatskiy2016optimal,flammarion2015averaging,SuBC16} has provided different intuitive interpretations of Nesterov acceleration~\cite{Nesterov1983}, the first work to demonstrate acceleration in its full generality. Nemirovski acceleration~\cite{nemirovskii1983problem,nemirovski-line-search-acc,nemirovski-1D-line-search} is lesser known~\cite{nemirovski-acceleration}, and significantly less attention has been devoted to providing intuition on how it is achieved and to relating it to Nesterov acceleration. Convergence of the CG method is typically analyzed using Chebyshev polynomials (see, e.g.~\cite{shewchuk1994introduction,vishnoi2013lx}). An exception is the analysis in~\cite{nemirovskii1983problem}, which motivated the accelerated method of Nemirovski~\cite{nemirovski-1D-line-search}. The analysis presented here is arguably more intuitive.

There are two recent works we are aware of that have established connections between CG and accelerated methods. A recent work of Scieur~\cite{scieur2019generalized} provided an insightful unifying analysis of different quasi-Newton methods and related them to CG. As the focus in~\cite{scieur2019generalized} was on quasi-Newton methods, Nemirovski and Nesterov acceleration were not considered. Also related to our work are the recent results of Drori and Taylor~\cite{Drori2019}, who utilized a computer-assisted approach to study a  ``CG-like'' method and its corresponding (tight) worst-case convergence guarantees. This systematic approach also led to a few other accelerated methods for which the exact same proofs are valid. Interestingly, the automated approach from~\cite{Drori2019} leads to two sufficient conditions that an algorithm needs to satisfy to achieve accelerated convergence, similar to the conditions obtained directly from \adgt~in this work (see Theorem~\ref{thm:accelerated-methods}).

 \subsection{Preliminaries}
 
 We consider the problem $\min_{\vx \in \rr^n} f(\vx),$ where $f: \rr^n \rightarrow \rr$ is a convex differentiable function. Throughout the note, $\|\cdot\|$ denotes the standard Euclidean norm, and $\innp{\cdot, \cdot}$ denotes the inner product. We will further be assuming throughout that $f$ is $L$-smooth and $\mu$-strongly convex, possibly with $\mu = 0$ (in which case it is just convex). The definitions of smoothness and strong convexity are provided below, for completeness.
 
 \begin{definition}\label{def:smoothness}
 Given $L \in \rr_+,$ a continuously differentiable function $f: \rr^n \to \rr$ is said to be $L$-smooth, if $\forall \vx, \vy \in \rr^n:$
 $$
    f(\vy) \leq f(\vx) + \innp{\nabla f(\vx), \vy - \vx} + \frac{L}{2}\|\vy - \vx\|^2.
 $$
 \end{definition}
 
 A simple implication of Definition~\ref{def:smoothness} is that for $\vy = \vx - \frac{1}{L}\nabla f(\vx):$
 $$
    f(\vy) \leq f(\vx) - \frac{1}{2L}\|\nabla f(\vx)\|^2.
 $$
 
 \begin{definition}\label{def:strong-cvxity}
 Given $\mu \in \rr_+,$ a continuously differentiable function $f: \rr^n \to \rr$ is said to be $\mu$-strongly convex, if $\forall \vx, \vy \in \rr^n:$
 $$
    f(\vy) \geq f(\vx) + \innp{\nabla f(\vx), \vy - \vx} + \frac{\mu}{2}\|\vy - \vx\|^2.
 $$
 \end{definition}
 Observe that, setting $\vy = \vx^*$ and minimizing the right-hand side in Definition~\ref{def:strong-cvxity} over $\vy,$ we have, $\forall \vx$:
 $$
    f(\vx^*) \geq f(\vx) - \frac{1}{2\mu}\|\nabla f(\vx)\|^2.
 $$
 
\paragraph{Method of Conjugate Gradients}
In the standard setting of the method of conjugate gradients \eqref{eq:CG}, we want to minimize a quadratic function $f(\vx) = \frac{1}{2}\vx^T\mA\vx - \vb^T\vx$ over $\rr^n$, where $\mA \in \rr^{n \times n}$ is a positive semidefinite matrix.
 Observe that the gradient of $f$ at point $\vx \in \rr^n$ can be expressed as $\nabla f(\vx) = \mA \vx - \vb.$ 
 Let $\mu \geq 0$ denote the minimum eigenvalue of $\mA$ and $L > 0$ denote its maximum eigenvalue. Then $f(\cdot)$ is $\mu$-strongly convex and $L$-smooth.
 
 We start by reviewing the description and some basic properties of the method of conjugate gradients applied to a quadratic function and implemented with infinite precision. 
 %
 To state the method of conjugate gradients, we need the following definition of linear Krylov subspaces:
 \begin{equation}\label{eq:def-krylov-subspaces}
     \cl_k = \mathrm{Lin}\{\mA(\vx_0 - \vx^*), \mA^2(\vx_0 - \vx^*), \dots, \mA^k (\vx_0 - \vx^*)\}.
 \end{equation}
The sequence of points generated by the method of conjugate gradients can be described as:
\begin{equation}\label{eq:CG}\tag{CG}
    \vy_{k} = \argmin_{\vx \in \vx_0 + \cl_k} f(\vx), \quad \forall k \geq 1.
\end{equation}
 A useful property of~\eqref{eq:CG} used in our analysis is the following (see, e.g.,  Lemma 1.3.1 in~\cite{nesterov2018lectures}):
 \begin{lemma}\label{lemma:alt-krylov-def}
 For any $k \geq 1,$ $\cl_k = \mathrm{Lin}\{\nabla f(\vy_0), \dots, \nabla f(\vy_{k-1})\}.$
 \end{lemma}
 Using this lemma, it is not hard to show that for all $k, i$ such that $k \neq i,$ it holds $\innp{\nabla f(\vy_k), \nabla f(\vy_i)} = 0.$ This follows simply by noting that $\vy_k = \argmin_{\bm{\lambda}\in \rr^{k-1}} f(\vx_0 + \sum_{i=1}^{k-1} \lambda_i \nabla f(\vy_i))$, which implies $\frac{\partial f(\vy_k)}{\partial \lambda_i} = 0.$  As there can be at most $n$ non-zero orthogonal vectors in $\rr^n,$~\eqref{eq:CG} converges after at most $n$ iterations.
 
Although the description of conjugate gradients as stated in~\eqref{eq:CG} may suggest that its iterations would be computationally intensive, there are various ways of implementing the iterations efficiently~\cite{nesterov2018lectures,vishnoi2013lx,shewchuk1994introduction}, possibly only using a constant number of matrix-vector multiplications. 

 \section{Generic Acceleration: the Approximate Duality Gap Technique}\label{sec:adgt}

 The Approximate Duality Gap Technique (\adgt)~\cite{thegaptechnique} provides a unified mathematical framework for the convergence analysis of first-order methods and is closely related to the powerful estimate sequence technique of Nesterov. The basic premise behind this approach is that every first-order algorithm is implicitly or explicitly constructing estimates of the optimal value $f(\vx^*),$ and, consequently, of the optimality gap $f(\vx_k) - f(\vx^*),$ given a candidate solution $\vx_k.$  
 More formally, at every iteration $k$, based on its history, the algorithm constructs an upper bound $U_k$ and a lower bound $L_k$ such that $U_k \geq f(\vx_k) \geq f(\vx^*) \geq L_k.$ Together, these estimates yield a notion of approximate duality gap $G_k = U_k - L_k,$ which bounds the error of the algorithm at iteration $k$. The upshot of \adgt~is that all different first-order methods can be derived as minimizing different notions of approximate optimality (or duality, see~\cite{thegaptechnique}) gap $G_k,$ based on combining a small number of choices for $U_k$ and $L_k.$ 
 
 For the purpose of this note, we confine ourselves to the unconstrained minimization of a convex differentiable function $f(\vx)$ that is $L$-smooth and $\mu$-strongly convex (possibly with $\mu = 0$) with respect to the Euclidean norm $\|\cdot\|$, given access to a first-order oracle for $f.$ (For generalizations to constrained optimization and other normed spaces, see~\cite{thegaptechnique,AXGD,cohen2018acceleration}.) Letting $\vx_k$ be the query point at iteration $k$, accelerated methods for this setting  are obtained by considering the following estimates.
 
 \paragraph{Upper Bound} We define the upper bound $U_k$ as $U_k = f(\vy_k),$ where $\vy_k$ is a point constructed based on the previous gradient query points $\{\vx_i\}_{i=0}^k$ and the gradient oracle answers $\{\nabla f(\vx_i)\}_{i=0}^k.$ In what follows, we will assume that the point $\vy_k$ is such that
 $$
    f(\vy_k) \leq f(\vx_k) - \frac{1}{2L}\|\nabla f(\vx_k)\|^2.
 $$
  Because the function $f$ is smooth, this can simply be achieved by choosing $\vy_k = \vx_k - \frac{1}{L}\nabla f(\vx_k).$ 

 \paragraph{Lower Bound}
 Each queried gradient $\nabla f(\vx_i)$ yields a lower bound on the function $f$ in the form of
 \begin{equation}\label{eq:quadlwbd}
 \forall \vu \in \mathbb{R}^n, \;\; f(\vu) \geq f(\vx_i) + \innp{\nabla f(\vx_i), \vu - \vx_i} + \frac{\mu}{2} \|\vu - \vx_i\|^2.
 \end{equation}
 Following the arguments from~\cite{thegaptechnique}, we assign to each iteration $k$ a measure $a_k > 0$ and denote by $A_k = \sum_{i=0}^k a_i$ the cumulative measure of all iterations up to $k.$ 
 We can then consider the lower bound obtained by averaging the bound for each $\vx_i$ in Equation~\eqref{eq:quadlwbd} with weight proportional to $a_i$ and adding a regularization term $\frac{\mu_0}{2A_k} \|\vu - \vx_0\|^2,$ 
 where $\mu_0 = L-\mu.$
 \begin{equation}
 \forall \vu \in \mathbb{R}^n, \;\; f(\vu) + \frac{\mu_0}{2A_k} \|\vu - \vx_0\|^2\geq \sum_{i=0}^k a_i \left(f(\vx_i) + \innp{\nabla f(\vx_i), \vu - \vx_i} + \frac{\mu}{2} \|\vu - \vx_i\|^2\right) + \frac{\mu_0}{2A_k} \|\vu - \vx_0\|^2.
 \end{equation}
 Taking  $\vu = \vx^*$ on the left-hand side and minimizing over $\vu$ on the right yields $f(\vx^*) \geq L_k$, where:
  \begin{equation*}
     L_k := \frac{\sum_{i=0}^k a_i f(\vx_i) + \min_{\vu \in \rr^n}\{\sum_{i=0}^k a_i (\innp{\nabla f(\vx_i), \vu - \vx_i} + \frac{\mu}{2}\|\vu - \vx_i\|^2)+ \frac{\mu_0}{2}\|\vu - \vx_0\|^2\}}{A_k} - \frac{\mu_0}{2A_k}\|\vx_0 - \vx^*\|^2.
 \end{equation*}
 Further, we denote by $\vv_k$ the minimizer in the definition of $L_k$:
 \begin{equation}\label{eq:def-of-vk}
 \vv_k = \argmin_{\vu \in \rr^n}\Big\{\sum_{i=0}^k a_i (\innp{\nabla f(\vx_i), \vu - \vx_i} + \frac{\mu}{2}\|\vu - \vx_i\|^2)+ \frac{\mu_0}{2}\|\vu - \vx_0\|^2\Big\}.
 \end{equation}
 Observe that we can explicitly write $\vv_k$ as:
 \begin{equation}\label{eq:explicit-v-k}
 \begin{aligned}
     \vv_k &= \frac{\mu_0 \vx_0 + \mu \sum_{i=0}^k a_i \vx_i - \sum_{i=0}^k a_i \nabla f(\vx_i)}{\mu_0 + \mu A_k}\\
     &= \frac{\mu_0 + \mu A_{k-1}}{\mu_0 + \mu A_k}\vv_{k-1} + \frac{\mu a_k}{\mu_0 + \mu A_k} \vx_k - \frac{a_k}{\mu_0 + \mu A_k}\nabla f(\vx_k).
     \end{aligned}
 \end{equation}
 
 \subsection{The Approximate Duality Gap and Its Evolution}
 
 The estimate $G_k$ for the duality gap at iteration $k$ is defined as $G_k = U_k - L_k.$ By the construction of $U_k$ and $L_k$, we have $U_k = f(\vy_k)$ and $L_k \leq f(\vx^*)$, so that $G_k \geq f(\vy_k) - f(\vx^*)$ bounds the error of $\vy_k$ from optimum.
Following the approximate duality gap technique~\cite{thegaptechnique}, we will construct algorithms for which $G_k$ goes to $0$ as $k$ grows by showing that $A_k G_k$ is non-increasing with the iteration count $k.$ This immediately implies that 
$$
f(\vy_k) - f(\vx^*) \leq G_k \leq \frac{A_0 G_0}{A_k},
$$
where $A_0 G_0$ is a fixed quantity related to the initial gap and $A_k \to \infty$ as $k \to \infty.$ 
Based on the constructed upper and lower bound sequences $U_k$ and $L_k,$ the initial gap estimate can be bounded as follows.
\begin{proposition}\label{prop:initial-gap}
Let $\vx_0 \in \rr^n$ be an arbitrary initial point and let the gap estimate $G_k$ be constructed as described in this section. Then:
$$
A_0 G_0 \leq \frac{L-\mu}{2}\|\vx^* - \vx_0\|^2.
$$
\end{proposition}
\begin{proof}
 Recall that the point $\vy_0$ is chosen so that $f(\vy_0) \leq f(\vx_0) - \frac{1}{2L}\|\nabla f(\vx_0)\|^2$ and that $\mu_0 = L- \mu.$ As $a_0 = A_0,$ using the definitions of $U_0$ and $L_0,$ we have:
 \begin{align*}
     A_0 G_0  &\leq - \frac{A_0}{2L}\|\nabla f(\vx_0)\|^2 - \min_{\vu \in \rr^n}\Big\{ a_0 \innp{\nabla f(\vx_0), \vu - \vx_0} + \frac{L}{2}\|\vu - \vx_0\|^2\Big\} + \frac{L-\mu}{2}\|\vx^* - \vx_0\|^2\\
     &\leq \frac{L-\mu}{2}\|\vx^* - \vx_0\|^2,
 \end{align*}
 as claimed.
\end{proof}

To carry out this approach, we will now bound the change $A_{k} G_{k} - A_{k-1} G_{k-1}$ at iteration $k$ as a function of the sequences of points $\{\vx_i, \, \vv_i, \, \vy_i\}_{i=0}^k.$ 
In the rest of the note, we will then demonstrate how different accelerated methods choose the query point so that $A_{k} G_{k} - A_{k-1} G_{k-1} \leq 0$ at every iteration $k$.

\begin{lemma} \label{lemma:gap-change}
    Let $\{\vx_i\}_{i=0}^k$ be an arbitrary sequence of points from $\rr^n,$ let $\vv_i$ be defined according to~\eqref{eq:def-of-vk} for $i \in \{0, \dots, k\},$ and $\vy_i$ be such that $f(\vy_i) \leq f(\vx_i) - \frac{1}{2L}\|\nabla f(\vx_i)\|^2,$ for all $i \in \{0, \dots, k\}.$ Define $a_k' = \frac{a_k(\mu_0 + \mu A_{k-1})}{\mu_0 + \mu A_k}.$   Then: 
 \begin{align*}
    A_k G_k - A_{k-1}G_{k-1} \leq & \Big(\frac{{a_k}^2}{2(\mu_0 + \mu A_k)} - \frac{A_k}{2L}\Big)\|\nabla f(\vx_k)\|^2\\
    & + \innp{\nabla f(\vx_k), (A_{k-1} + a_k') \vx_{k} - A_{k-1} \vy_{k-1} - a_k' \vv_{k-1}}.
 \end{align*}
\end{lemma}
%
\begin{proof}
 We start by bounding the change in the lower bound sequence. Denote by:
    $$
        m_k(\vu) = \sum_{i=0}^k a_i (\innp{\nabla f(\vx_i), \vu - \vx_i} + \frac{\mu}{2}\|\vu - \vx_i\|^2)+ \frac{\mu_0}{2}\|\vu - \vx_0\|^2
    $$
 the function inside the minimum in the definition of $L_k.$ We are interested in comparing $\min_\vu m_{k-1} (\vu) = m_{k-1}(\vv_{k-1})$ with $\min_\vu m_k(\vu) = m_k(\vv_k)$. We first note the direct relation between $m_k(\cdot)$ and $m_{k-1}(\cdot)$
    $$
         m_{k}(\vv_k) = m_{k-1}(\vv_k) + a_k \innp{\nabla f(\vx_k), \vv_k - \vx_k} + a_k \frac{\mu}{2}\|\vv_k - \vx_k\|^2.
    $$
Because $m_{k-1}(\vu)$ is a sum of linear and quadratic terms in $\vu$, where the total weight of the quadratic terms is $(\mu_0 + \mu A_{k-1})$, and because it is minimized at $\vv_{k-1}$, we have:
$$
    m_{k-1}(\vv_k) = m_{k-1}(\vv_{k-1}) + \frac{\mu_0 + A_{k-1}\mu}{2} \|\vv_k - \vv_{k-1}\|^2. 
$$
Hence, it follows that:
    \begin{equation}\label{eq:lwbrec}
        m_{k}(\vv_k) = m_{k-1}(\vv_{k-1}) + \frac{\mu_0 + A_{k-1}\mu}{2} \|\vv_k - \vv_{k-1}\|^2 + a_k \innp{\nabla f(\vx_k),    \vv_k - \vx_k} + a_k \frac{\mu}{2}\|\vv_k - \vx_k\|^2.
    \end{equation}
Applying Jensen inequality to the quadratic terms in the right-hand side of last equation:
    \begin{equation*}
        m_k(\vv_k) - m_{k-1}(\vv_{k-1}) \geq \frac{\mu_0 + \mu A_k}{2}\Big\|\vv_k - \frac{\mu_0 + \mu A_{k-1}}{\mu_0 + \mu A_k} \vv_{k-1} - \frac{\mu a_k}{\mu_0 + \mu A_k}\vx_k\Big\|^2 + a_k \innp{\nabla f(\vx_k), \vv_k - \vx_k}.
    \end{equation*}
Recalling from Eq.~\eqref{eq:explicit-v-k} that $\vv_k = \frac{\mu_0 + \mu A_{k-1}}{\mu_0 + \mu A_k}\vv_{k-1} + \frac{\mu a_k}{\mu_0 + \mu A_k} \vx_k - \frac{a_k}{\mu_0 + \mu A_k}\nabla f(\vx_k),$ we further have:
    \begin{align*}
         m_k(\vv_k) - m_{k-1}(\vv_{k-1}) &\geq  - \frac{{a_k}^2}{2(\mu_0 + \mu A_k)}\|\nabla f(\vx_k)\|^2 + a_k \frac{\mu_0 + \mu A_{k-1}}{\mu_0 + \mu A_k}\innp{\nabla f(\vx_k), \vv_{k-1} - \vx_k}\\
         &= -\frac{{a_k}^2}{2(\mu_0 + \mu A_k)}\|\nabla f(\vx_k)\|^2 + a_k' \innp{\nabla f(\vx_k), \vv_{k-1} - \vx_k}.
     \end{align*}
 We thus have the following:
    \begin{align*}
         A_k L_k - A_{k-1}L_{k-1} &= a_k f(\vx_k) + m_{k}(\vv_k) - m_k(\vv_{k-1})\\
        &\geq a_k f(\vx_k)  - \frac{{a_k}^2}{2(\mu_0 + \mu A_k)}\|\nabla f(\vx_k)\|^2 +  a_k' \innp{\nabla f(\vx_k), \vv_{k-1} - \vx_k}.
    \end{align*}
 On the other hand, using convexity of $f(\cdot)$ and that $f(\vy_k) \leq f(\vx_k) - \frac{1}{2L}\|\nabla f(\vx_k)\|^2$, we can bound the change in the upper bound sequence as:
    \begin{align*}
         A_k U_k - A_{k-1}U_{k-1} &= A_k f(\vy_{k}) - A_{k-1} f(\vy_{k-1})\\
            &= A_k (f(\vy_k) - f(\vx_k)) - A_{k-1}(f(\vy_{k-1}) - f(\vx_k)) + a_k f(\vx_k)\\
        & \leq -\frac{A_k}{2L}\|\nabla f(\vx_k)\|^2 + A_{k-1}\innp{\nabla f(\vx_k), \vx_k - \vy_{k-1}} + a_k f(\vx_k).
     \end{align*}
 Combining with the change in the lower bound sequence, we finally have:
    \begin{align*}
          A_k G_k - A_{k-1}G_{k-1} \leq & \Big(\frac{{a_k}^2}{2(\mu_0 + \mu A_k)} - \frac{A_k}{2L}\Big)\|\nabla f(\vx_k)\|^2\\
          &+ \innp{\nabla f(\vx_k), (A_{k-1} + a_k') \vx_{k} - A_{k-1} \vy_{k-1} - a_k' \vv_{k-1}},
    \end{align*}
    as claimed.
 \end{proof}
 
 Using Proposition~\ref{prop:initial-gap} and Lemma~\ref{lemma:gap-change}, we can deduce  sufficient conditions on the gradient query points $\{\vx_i\}_{i=0}^k$ and step sizes $\{a_i\}_{i=0}^k$ that lead to the optimal convergence rates of accelerated algorithms.
 
 \begin{theorem}\label{thm:accelerated-methods}
  Let the sequence of gradient query points $\{\vx_i\}_{i=0}^k$ be such that, for all $i \in \{1,\dots, k\}$:
  \begin{equation}\label{eq:cond-for-x}
    \innp{\nabla f(\vx_i), (A_{i-1} + a_i') \vx_{i} - A_{i-1} \vy_{i-1} - a_i' \vv_{i-1}} = 0,
  \end{equation}
  where points $\vv_i$ are defined by~\eqref{eq:def-of-vk}, points $\{\vy_i\}_{i=0}^k$ are such that $f(\vy_i) \leq f(\vx_i) - \frac{1}{2L}\|\nabla f(\vx_i)\|^2$, $a_i' = \frac{a_i(\mu_0 + \mu A_{i-1})}{\mu_0 + \mu A_i},$ and the sequence of scalars $\{a_i\}_{i=0}^k$ is defined by $a_0 = 1,$ $\frac{{a_i}^2}{A_i(\mu_0 + \mu A_i)} = \frac{1}{L},$ where $A_i = \sum_{j=0}^i a_j.$ Then:
   \begin{equation*}
     f(\vy_{k}) - f(\vx^*) \leq \min\bigg\{\frac{4 }{(k+1)(k+2)}, \Big(1-\sqrt{\frac{\mu}{L}}\Big)^{k}\bigg\}\frac{(L-\mu)\|\vx^* - \vx_0\|^2}{2}.
 \end{equation*}
 \end{theorem}
 \begin{proof}
  The assumptions of the theorem guarantee that, using Lemma~\ref{lemma:gap-change}, $A_k G_k - A_{k-1}G_{k-1} \leq 0.$ Hence, combining with Proposition~\ref{prop:initial-gap}, we have:
  $$
    f(\vy_k) - f(\vx^*) \leq G_k \leq \frac{A_0 G_0}{A_k} \leq \frac{(L-\mu)\|\vx^* - \vx_0\|^2}{2 A_k}.
  $$
  To complete the proof, it remains to argue about the growth of the sequence of positive numbers $\{A_i\}_{i=0}^k.$ First, when $\mu = 0,$ we have that $\mu_0 = L,$ and the sequence needs to satisfy $\frac{{a_i}^2}{A_i} = 1.$ Because this sequence dominates the sequence $\{B_i\}_{i=0}^k,$ where $B_i = \sum_{j=0}^i b_j$ and $b_j = \frac{j+1}{2}$ (as $\frac{{b_j}^2}{B_j} < 1$), we have that $A_k \geq B_k = \frac{(k+1)(k+2)}{4}.$ This gives the first term from the minimum in the statement of the theorem. 
  
  For the second term in the minimum, assume that $\mu > 0.$ Then, as $\frac{{a_i}^2}{A_i(\mu_0 + \mu A_i)} \leq \frac{{a_i}^2}{\mu {A_i}^2},$ the sequence of numbers $\{A_i\}_{i=0}^k$ dominates the sequence $\{B_i\}_{i=0}^k,$ defined by $B_i = \sum_{j=0}^i b_j,$ $b_0 = 1,$ and $\frac{{b_j}^2}{{B_j}^2} = \frac{\mu}{L}$ for $j \geq 1.$ Hence, $A_k \geq B_k = (1- \sqrt{\frac{\mu}{L}})^{-k},$ completing the proof.
 \end{proof}

\subsection{Accelerated Methods in the Approximate Duality Gap View}
 
In the rest of the note, we show how the convergence analysis of different accelerated methods follows as an application of Theorem~\ref{thm:accelerated-methods}. In what follows, we assume that the sequences $\{a_i\}_{i=0}^k,$ $A_i = \sum_{j=0}^i a_j$ are chosen according to Theorem~\ref{thm:accelerated-methods}. Thus, the only thing that remains to be satisfied is the choice of gradient query points $\{\vx_i\}_{i=0}^k$ and solution points $\{\vy_i\}_{i=0}^k$.

 \paragraph{Standard Nesterov Acceleration} 
 
 Nesterov accelerated method~\cite{Nesterov1983} as described in the textbook by the same author~\cite{nesterov2018lectures} defines $\vx_i,\, \vy_i$ by:
 \begin{equation}\notag
     \begin{gathered}
        \vx_i = \frac{A_{i-1}}{a_i' + A_{i-1}} \vy_{i-1} + \frac{a_i'}{a_i' + A_{i-1}} \vv_{i-1},\\
        \vy_i = \vx_i - \frac{1}{L}\nabla f(\vx_i).
     \end{gathered}
 \end{equation}
 It is immediate that this choice of points satisfies the assumptions of Theorem~\ref{thm:accelerated-methods}, and the accelerated convergence bound follows.

\paragraph{Nemirovski Acceleration with a Plane Search} For smooth convex minimization (with $\mu = 0$), Nemirovski accelerated method with a plane search~\cite{nemirovskii1983problem,nemirovski-line-search-acc} can be stated as\footnote{Note that some versions of Nemirovski's method minimize $f$ over $\vx = \vy_{i-1} + \alpha (\vy_{i-1}-\vx_0) + \beta \big(\frac{1}{L}\sum_{j=0}^{i-1} a_j \nabla f(\vx_j)\big)$. This change is irrelevant for the argument presented below, as it leads to $\innp{\nabla f(\vx_i), \vy_{i-1}-\vx_0} = 0$ and $\innp{\nabla f(\vx_i), \vv_{i-1}-\vx_0} = 0$. Same as argued below, this choice is sufficient for the condition from Eq.~\eqref{eq:cond-for-x} to hold.}:
\begin{equation}\notag
    \begin{gathered}
        \vx_i = \argmin \Big\{f(\vx): \vx = \alpha \vy_{i-1} + \beta \Big(\vx_0 - \frac{1}{L}\sum_{j=0}^{i-1} a_j \nabla f(\vx_j)\Big),\, \alpha, \beta \in \rr\Big\},\\
        \vy_i = \vx_i - \frac{1}{L}\nabla f(\vx_i).
    \end{gathered}
\end{equation}
 It is not hard to see (recalling the definition of $\vv_i$ from Eqs.~\eqref{eq:def-of-vk}, \eqref{eq:explicit-v-k} and $\mu_0 = L - \mu$), that this definition of $\vx_i$ in the special case of $\mu = 0$ precisely corresponds to
 $$
    \vx_i = \argmin \Big\{f(\vx): \vx = \alpha \vy_{i-1} + \beta \vv_{i-1},\, \alpha, \beta \in \rr\Big\}.
 $$
 Hence, using first-order optimality conditions in the definition of $\vx_i$, it follows that $\innp{\nabla f(\vx_i), \vy_{i-1}} = 0$ and $\innp{\nabla f(\vx_i), \vv_{i-1}} = 0,$ which is sufficient to satisfy the condition from Eq.~\eqref{eq:cond-for-x}, and, hence, Theorem~\ref{thm:accelerated-methods} applies.
 
 \paragraph{Nemirovski Acceleration with a Line Search} A further refinement of the method that replaces the plane search with a line search was also provided by Nemirovski~\cite{nemirovski-1D-line-search}, and can be stated as:
 \begin{equation}\notag
    \begin{gathered}
        \vx_i = \argmin \Big\{f(\vx): \vx = \vy_{i-1} + \beta \Big(\vx_0 - \vy_{i-1} - \frac{1}{L}\sum_{j=0}^{i-1} a_j \nabla f(\vx_j)\Big),\, \beta \in \rr\Big\},\\
        \vy_i = \vx_i - \frac{1}{L}\nabla f(\vx_i).
    \end{gathered}
\end{equation}
Again, recalling the definitions of $\vv_i$ and $\mu_0,$ in the case of $\mu=0,$ this choice of $\vx_i$ precisely corresponds to:
$$
    \vx_i = \argmin \Big\{f(\vx): \vx = (1-\beta) \vy_{i-1} + \beta \vv_{i-1},\, \beta \in \rr\Big\}.
$$
 Using the first-order optimality condition in the definition of $\vx_i,$ it follows that $\innp{\nabla f(\vx_i), \vv_{i-1} - \vy_{i-1}} = 0.$ A simple calculation reveals that this is sufficient to ensure that the condition from Eq.~\eqref{eq:cond-for-x} holds, which leads to the result from Theorem~\ref{thm:accelerated-methods}.
 
 \paragraph{Method of Conjugate Gradients} Finally, to obtain the accelerated convergence bound for~\eqref{eq:CG}, here we assume that $f(\vx) = \frac{1}{2}\innp{\mA \vx, \vx} - \innp{\vb, \vx},$ for some positive semidefinite matrix $\mA$, whose minimum and maximum eigenvalues are $\mu$ and $L,$ respectively. Given that $\vy_i = \argmin_{\vx \in \vx_0 + \cl_i}f(\vx),$ it follows that $f(\vy_i) \leq f(\vx) - \frac{1}{2L}\|\nabla f(\vx)\|^2,$ for any point $\vx \in \vx_0 + \cl_{i-1}.$ To apply Theorem~\ref{thm:accelerated-methods}, it suffices to show that $\vv_i \in \vx_0 + \cl_{i},$ $\forall i.$ For then, \emph{any} choice of $\vx_i$ given by the accelerated algorithms described above would belong to $\vx_0 + \cl_{i-1},$ which, given that $f(\vy_i) \leq f(\vx_i) - \frac{1}{2L}\|\nabla f(\vx_i)\|^2,$ leads to Theorem~\ref{thm:accelerated-methods}.
 
 The proof that $\vv_i \in \vx_0 + \cl_{i}$ is by induction on $i.$ Clearly, the claim holds initially, as $\vv_0 = \vx_0 - \frac{1}{L}\nabla f(\vx_0) \in \vx_0 + \cl_0.$ Suppose that the claim holds up to iteration $i-1.$ Then, for any of the choices of $\vx_{i}$ described above, we have that $\vx_{i}$ is a linear combination of $\vy_{i-1}$ and $\vv_{i-1}$; as both $\vy_{i-1}$ and $\vv_{i-1}$ are from $\vx_0 + \cl_{i-1},$ it follows that $\vx_i \in \vx_0 +  \cl_{i-1}.$ Observe that $\nabla f(\vx_i) = \mA \vx_i - \vb \in \cl_{i}.$ Recalling the explicit definition of $\vv_i$ from Eq.~\eqref{eq:explicit-v-k}, $\vv_i$ is the sum of a convex combination of $\vv_{i-1}$ and $\vx_i$, and a constant multiple of $\nabla f(\vx_i).$ As such, $\vv_i$ must belong to $\cl_{i},$ as claimed.
 
 Let us make one final observation. For~\eqref{eq:CG}, as each step minimizes the function over a direction that is orthogonal to the previous directions (see Lemma~\ref{lemma:alt-krylov-def}), we have that, at iteration $i+1$:
 $$
    f(\vy_i) - \frac{1}{2\ell_i}\|\nabla f(\vy_i)\|^2 \geq f(\vy_{i+1}) \geq f(\vx^*) \geq f(\vy_i) - \frac{1}{2\mu_i}\|\nabla f(\vy_i)\|^2,
 $$
 where $\mu_i$ and $\ell_i$ are the minimum and the maximum eigenvalue of $\mA$ over the complement of the subspace $\cl_i.$ Once the complement of $\cl_i$ becomes one-dimensional (which happens after at most $n-1$ iterations), it must be $\mu_i = \ell_i$, in which case $\vy_{i+1} = \vx^*.$ This is an alternative argument that shows that~\eqref{eq:CG} converges after at most $n$ iterations, using upper and lower bounds that are specific to quadratic functions and~\eqref{eq:CG}.

 \section{Conclusion}
 
 This note provides a simple and intuitive analysis of accelerated methods for smooth convex and smooth strongly convex minimization, including the method of conjugate gradients for quadratic programs, and using the Approximate Duality Gap Technique. Some interesting questions still remain. For example, it would be interesting to see whether \adgt~can be used to analyze the numerical stability of CG and suggest modifications that improve it.
 
 \section*{Acknowledgements}
 We thank Rasmus Kyng for pointing out a few typos in an earlier version of the note, and for his useful comments.  
 
 \bibliographystyle{abbrv}
 \bibliography{references.bib}

\begin{thebibliography}{10}

\bibitem{AO-survey-nesterov}
Z.~{Allen-Zhu} and L.~Orecchia.
\newblock Linear coupling: An ultimate unification of gradient and mirror
  descent.
\newblock In {\em Proc. ITCS'17}, 2017.

\bibitem{nemirovski-acceleration}
S.~Bubeck.
\newblock Nemirovski's acceleration, Jan. 2019.
\newblock
  \url{https://blogs.princeton.edu/imabandit/2019/01/09/nemirovskis-acceleration/}.

\bibitem{Bubeck2015}
S.~Bubeck, Y.~T. Lee, and M.~Singh.
\newblock {A geometric alternative to Nesterov's accelerated gradient descent}.
\newblock {\em arXiv preprint, arXiv:1506.08187}, 2015.

\bibitem{cohen2018acceleration}
M.~B. Cohen, J.~Diakonikolas, and L.~Orecchia.
\newblock On acceleration with noise-corrupted gradients.
\newblock In {\em Proc. ICML'18}, 2018.

\bibitem{AXGD}
J.~Diakonikolas and L.~Orecchia.
\newblock Accelerated extra-gradient descent: A novel, accelerated first-order
  method.
\newblock In {\em Proc. ITCS'18}, 2018.

\bibitem{thegaptechnique}
J.~Diakonikolas and L.~Orecchia.
\newblock The approximate duality gap technique: A unified theory of
  first-order methods.
\newblock {\em SIAM J. Optimiz.}, 29(1):660--689, 2019.

\bibitem{Drori2019}
Y.~Drori and A.~B. Taylor.
\newblock Efficient first-order methods for convex minimization: A constructive
  approach.
\newblock {\em Math. Program.}, Jun 2019.

\bibitem{drusvyatskiy2016optimal}
D.~Drusvyatskiy, M.~Fazel, and S.~Roy.
\newblock An optimal first order method based on optimal quadratic averaging.
\newblock {\em SIAM J. Optimiz.}, 28(1):251--271, 2018.

\bibitem{flammarion2015averaging}
N.~Flammarion and F.~Bach.
\newblock From averaging to acceleration, there is only a step-size.
\newblock In {\em Proc. COLT'15}, 2015.

\bibitem{Hestenes1952}
M.~R. Hestenes and E.~Stiefel.
\newblock Methods of conjugate gradients for solving linear systems.
\newblock {\em J. Res. Natl. Bur. Stand. (U.S.)}, 49(6):409--436, 1952.

\bibitem{nemirovski-1D-line-search}
A.~Nemirovskii and D.~Yudin.
\newblock Information-based complexity of mathematical programming (in
  {R}ussian).
\newblock {\em Izvestia AN SSSR, Ser. Tekhnicheskaya Kibernetika (Translated
  as: Engineering Cybernetics. Soviet J. Computer \& Systems Sci.)}, (1), 1983.

\bibitem{nemirovskii1983problem}
A.~Nemirovskii and D.~B. Yudin.
\newblock {\em Problem complexity and method efficiency in optimization}.
\newblock Wiley, 1983.

\bibitem{nemirovski-line-search-acc}
A.~S. Nemirovskii.
\newblock Orth-method for smooth convex optimization (in {R}ussian).
\newblock {\em Izvestia AN SSSR, Ser. Tekhnicheskaya Kibernetika (Translated
  as: Engineering Cybernetics. Soviet J. Computer \& Systems Sci.)},
  (2):937--947, 1982.

\bibitem{Nesterov1983}
Y.~Nesterov.
\newblock {A method of solving a convex programming problem with convergence
  rate $O(1/k^2)$}.
\newblock {\em Soviet Mathematics Doklady}, 27(2):372--376, 1983.

\bibitem{nesterov2005smooth}
Y.~Nesterov.
\newblock Smooth minimization of non-smooth functions.
\newblock {\em Math. Program.}, 103(1):127--152, 2005.

\bibitem{nesterov2018lectures}
Y.~Nesterov.
\newblock {\em Lectures on convex optimization}.
\newblock Springer, 2018.

\bibitem{scieur2019generalized}
D.~Scieur.
\newblock Generalized framework for nonlinear acceleration.
\newblock {\em arXiv preprint arXiv:1903.08764}, 2019.

\bibitem{Scieur2017}
D.~Scieur, V.~Roulet, F.~Bach, and A.~D'Aspremont.
\newblock Integration methods and accelerated optimization algorithms.
\newblock In {\em Proc. NIPS'17}, 2017.

\bibitem{shewchuk1994introduction}
J.~R. Shewchuk.
\newblock {\em An introduction to the conjugate gradient method without the
  agonizing pain}.
\newblock Carnegie-Mellon University, Department of Computer Science, 1994.

\bibitem{SuBC16}
W.~Su, S.~Boyd, and E.~J. Candes.
\newblock A differential equation for modeling {Nesterov}'s accelerated
  gradient method: Theory and insights.
\newblock {\em J. Mach. Learn. Res.}, 17(153):1--43, 2016.

\bibitem{vishnoi2013lx}
N.~K. Vishnoi.
\newblock {$Lx= b:$ Laplacian} solvers and their algorithmic applications.
\newblock {\em Foundations and Trends{\textregistered} in Theoretical Computer
  Science}, 2013.

\bibitem{wibisono2016variational}
A.~Wibisono, A.~C. Wilson, and M.~I. Jordan.
\newblock A variational perspective on accelerated methods in optimization.
\newblock In {\em Proc. Natl. Acad. Sci. U.S.A.}, 2016.

\end{thebibliography}
\end{document}